\documentclass{conm-p-l}

\usepackage{amssymb}

\usepackage{upref}

\copyrightinfo{2009}{Charles F. Doran, Andrey Y. Novoseltsev}

\newtheorem{theorem}{Theorem}[section]
\newtheorem{corollary}[theorem]{Corollary}
\newtheorem{lemma}[theorem]{Lemma}
\newtheorem{proposition}[theorem]{Proposition}

\theoremstyle{definition}
\newtheorem{definition}[theorem]{Definition}
\newtheorem{example}[theorem]{Example}

\theoremstyle{remark}
\newtheorem{remark}[theorem]{Remark}

\numberwithin{equation}{section}

\newcommand{\PP}{\ensuremath{\mathbb{P}}}
\newcommand{\RR}{\ensuremath{\mathbb{R}}}
\newcommand{\ZZ}{\ensuremath{\mathbb{Z}}}

\newcommand{\cF}{\ensuremath{\mathcal{F}}}
\newcommand{\cP}{\ensuremath{\mathcal{P}}}

\DeclareMathOperator{\Conv}{Conv}
\DeclareMathOperator{\Hom}{Hom}
\DeclareMathOperator{\rk}{rk}

\newcommand{\abs}[1]{\left|#1\right|}
\newcommand{\set}[1]{\left\{#1\right\}}
\newcommand{\sumfrac}[2]{\genfrac{}{}{0pt}{1}{#1}{#2}}

\begin{document}

\title[Hodge Numbers of CICY Threefolds in Toric Varieties]{
Closed Form Expressions for Hodge Numbers\\
of Complete Intersection Calabi-Yau Threefolds\\
in Toric Varieties}

\author{Charles F. Doran}
\address{
Department of Mathematical and Statistical Sciences\\
632 CAB, University of Alberta\\
Edmonton, Alberta T6G 2G1\\
Canada}
\email{doran@math.ualberta.ca}
\thanks{The first author was supported in part by NSF Grant No. DMS-0821725 and NSERC-SAPIN Grant \#371661}

\author{Andrey Y. Novoseltsev}
\address{
Department of Mathematical and Statistical Sciences\\
632 CAB, University of Alberta\\
Edmonton, Alberta T6G 2G1\\
Canada}
\email{novoseltsev@math.ualberta.ca}


\date{October 4, 2009}

\begin{abstract}
We use Batyrev-Borisov's formula for the generating function of stringy Hodge numbers of Calabi-Yau varieties realized as complete intersections in toric varieties in order to get closed form expressions for Hodge numbers of Calabi-Yau threefolds in five-dimensional ambient spaces. These expressions involve counts of lattice points on faces of associated Cayley polytopes. Using the same techniques, similar expressions may be obtained for higher dimensional varieties realized as complete intersections of two hypersurfaces.\footnotemark
\end{abstract}
\footnotetext{In fact, at the time of publication of this article authors have obtained such formulas, they will be communicated in a subsequent work.}

\maketitle

\section{Introduction}

In~\cite{Batyrev1994} Batyrev obtained combinatorial formulas for Hodge numbers $h^{1,1}(X)$ and $h^{n-1,1}(X)$ of an $n$-dimensional Calabi-Yau variety $X$ arising as a hypersurface in a toric variety associated to a reflexive polytope. It is immediate from these formulas that $h^{1,1}(X)=h^{n-1,1}(X^\circ)$, where $X^\circ$ is Batyrev's mirror of $X$, and this equality suffices to show that mirror symmetry holds on the level of Hodge numbers for Calabi-Yau 3-folds. However, it is also important to consider higher dimensional Calabi-Yau varieties including singular ones.

Batyrev and Dais, motivated by ``physicists Hodge numbers'', introduced \emph{string-theoretic Hodge numbers}~\cite{BatyrevDais1996} for a certain class of singular varieties. The string-theoretic Hodge numbers coincide with the regular ones for smooth varieties and with regular Hodge numbers of a crepant desingularization if it exists. Later Batyrev also introduced \emph{stringy Hodge numbers}~\cite{Batyrev1998} for a different class of singular varieties. While stringy and string-theoretic Hodge numbers are not the same, they do agree for the varieties we will be dealing with in this paper, see~\cite{BorisovMavlyutov2003} for further details on relations between them.

Batyrev and Borisov were able to obtain a formula for the generating function of string-theoretic Hodge numbers in the case of complete intersections in toric varieties and show that this function has properties corresponding to mirror symmetry~\cite{BatyrevBorisov1996a}. While their formula can be used in practice for computing Hodge numbers (as it is done in software PALP~\cite{PALP}), it is recursive, takes significant time even on computers, and does not provide much qualitative information on particular Hodge numbers. This work was motivated by the desire to obtain, for complete intersections, formulas similar to those for hypersurfaces.

We were able to accomplish this goal in the case of two hypersurfaces intersecting in a five dimensional toric variety, see Theorem~\ref{thm:h11cigeneric} for arbitrary nef partitions and a simplified expression in Theorem~\ref{thm:h11cispecial} for the indecomposable ones. The algorithm allows one to get expressions for $h^{1,1}$ for the intersection of two hypersurfaces in a higher dimensional ambient space as  well.

\textbf{Acknowledgements.}
We would like to thank Victor Batyrev, Maximilian Kreuzer, Anvar Mavlyutov, John Morgan, and Raman Sanyal for inspiring discussions and references. We are grateful to our referee for his or her thorough review of our paper and pointing out quite a few possible improvements as well as typos.

It was also very beneficial for this project to be able to experiment with numerous  examples using  \texttt{lattice\_polytope} module~\cite{Novoseltsev2009} of the software system Sage~\cite{Sage}, which provides convenient access to PALP~\cite{PALP} as one of its features.

\section{Generating functions for stringy Hodge numbers} \label{sec:notation}

In this section we fix the notation, define a nef partition and the generating function for the stringy Hodge numbers of the associated variety. The exposition is based on~\cite{BatyrevBorisov1996a, BatyrevNill2007}, where one can also find further properties of the objects in question (the notation there is slightly different, as those authors work with faces of cones, not of supporting polytopes). Since this paper is mostly combinatorial, we will use the generating function to define the stringy Hodge numbers. 

Let $N\simeq\ZZ^n$ be a lattice of dimension $n$, $M=\Hom(N,\ZZ)$ be its dual lattice, $N_\RR=N\otimes_\ZZ\RR$ and $M_\RR=M\otimes_\ZZ\RR$ be the vector spaces spanned by these lattices. Let $\Delta\subset N_\RR$ be a reflexive polytope (a bounded convex polytope of dimension $n$ with vertices at lattice points and all facets having integral distance $1$ from the origin) and $\Delta^\circ\subset M_\RR$ be its polar (the convex hull of inner normals of facets of~$\Delta$, normalized to be primitive integral vectors~--- reflexivity of~$\Delta$ implies that $\Delta^\circ$ is also reflexive, hence the name). Let $\PP_\Delta$ be the toric variety corresponding to the fan spanned by faces of (a triangulation of the boundary of)  $\Delta$, see, for example,~\cite{CoxKatz1999} for details on constructing~$\PP_\Delta$.

Let the vertex set $V = V(\Delta)$ be partitioned into a disjoint union of subsets
\begin{gather*}
V = V_1 \coprod V_2 \coprod \ldots \coprod V_r\,,
\end{gather*}
with corresponding polytopes $\Delta_i = \Conv(V_i, 0)$. This decomposition determines a nef partition if the Minkowski sum $\Delta_1+\dots+\Delta_r$ is also a reflexive polytope which we will denote by  $\nabla^\circ$, for $\nabla\subset M_\RR$.

We say that this nef partition is \emph{indecomposable}, if the Minkowski sum of any proper subset of $\set{\Delta_i}_{i=1}^r$ is not a reflexive polytope in the generated sublattice. Decomposable nef partitions correspond to products of Calabi-Yau varieties, presented as complete intersections of smaller numbers of hypersurfaces in toric varieties of smaller dimensions.

The associated {\em Cayley polytope} $P^* \subset N_\RR \times \RR^r$ of dimension $n+r-1$ is given by
\begin{gather*}
P^* = \Conv(\Delta_1 \times e_1, \Delta_2 \times e_2, \ldots, \Delta_r \times e_r)\,,
\end{gather*}
where $\{e_i\}_{i=1}^r$ is the standard basis for $\ZZ^r \subset \RR^r$.  The Cayley polytope supports the {\em Cayley cone} $C^* \subset N_\RR \times \RR^r$ of dimension $n+r$. It is a reflexive Gorenstein cone of index $r$ ($rP^*$ is a reflexive polytope) with \emph{dual Cayley cone} $C\subset M_\RR\times \RR^r$ supported on the \emph{dual Cayley polytope} $P$. 

The intersections of $P$ with affine subspaces given by intersections of hyperplanes $x_i = 1,\ x_j=0$ for a fixed $i\in\{n+1,\dots,n+r\}$ and all $j\in\set{n+1,\dots,n+r}$, $j\neq i$, are polytopes $\nabla_1,\dots,\nabla_r$ corresponding to the dual nef partition such that $\nabla=\Conv(\nabla_1,\dots,\nabla_r)$ and $\Delta^\circ=\nabla_1+\dots+\nabla_r$.
These polytopes determine equations of hypersurfaces in $\PP_\Delta$ and their intersection is a (possibly decomposable and/or singular) Calabi-Yau variety $X$ of dimension $n-r$.

Faces of the cone $C$ with the inclusion relation form an {\em Eulerian poset} with
the minimal element the vertex at the origin and the maximal element $C$ itself.
It is convenient to use faces of $P$ to index elements of this poset, with $\emptyset$ and $P$ representing the vertex of $C$ and $C$ itself respectively. If $x$ is a face of $P$, we will denote by $x^\vee$ the dual face of $P^*$.

For any Eulerian poset $\cP$ with the minimum element $\hat{0}$ and the maximum element~$\hat{1}$, if $x,y\in\cP$ with $x\leqslant y$ we will use notation $\rk x$ for the rank of $x$, the length of the longest chain of element between $\hat{0}$ and $x$, $[x,y]=\set{z\in\cP : x\leqslant z\leqslant y}$ for the subposet of elements between $x$ and $y$ inclusively, and $d\cP=\rk\hat{1}$ for the rank of the poset.

If $x$ and $y$ are faces of $P$ with $x\subset y$, then $d[x,y]=\dim y-\dim x$, $\rk x=\dim x + 1$, and it is natural to define here $\dim\emptyset = -1$, since dimensions of faces of $P$ are less by one than the dimensions of corresponding faces of $C$.

\begin{definition}
Let $\cP$ be an Eulerian poset of rank $d$ with the minimal element $\hat{0}$ and the maximal one $\hat{1}$. For $d=0$ let $G_\cP=H_\cP=B_\cP=1$. For $d>0$ define polynomials $G_\cP,\ H_\cP(t)\in\ZZ[t]$ and $B_\cP(u,v)\in\ZZ[u,v]$ recursively by
\begin{align*}
H_\cP(t) &= \sum_{\hat{0}<x\leqslant\hat{1}} (t-1)^{\rk{x}-1} G_{[x,\hat{1}]}(t),\\
G_\cP(t) &= \tau_{<d/2} (1-t) H_\cP(t),
\end{align*}
where
\begin{align*}
\tau_{<d/2}\sum_{k=0}^\infty a_k t^k= \sum_{0\leqslant m<d/2} a_k t^k
\end{align*}
is the truncation operator, and
\begin{align*}
\sum_{\hat{0}\leqslant x\leqslant\hat{1}} B_{[\hat{0},x]}(u,v) u^{d-\rk x} G_{[x,\hat{1}]}(u^{-1}v)=G_\cP(uv).
\end{align*}
\end{definition}

\begin{proposition} \label{prop:Bproperties}
The $B_\cP$ polynomial defined above has the following properties:
\begin{enumerate}
\item The degree of $B_\cP(u,v)$ in $v$ is (strictly) less than $d\cP/2$.
\item If $d\cP\leqslant 2$, then $B_\cP(u,v)=(1-u)^{d\cP}$.
\item If $\cP$ is Eulerian poset associated to faces of a polygon with $k$ vertices/edges, then $d\cP=3$ and $B_\cP(u,v)=1+[k-(k-3)v](u^2-u)-u^3$.
\end{enumerate}
\end{proposition}

\begin{proof}
See~\cite{BatyrevBorisov1996a}, Examples 2.8, 2.9, and Proposition 2.10.
\end{proof}

\begin{definition}
Let $\cF\in N$ be a $d$-dimensional lattice polytope (or a $d$-dimensional face of a lattice polytope). Let $\ell(\cF)=\abs{\cF\cap N}$ be the number of lattice points inside~$\cF$. Let $\ell^*(\cF)$ be the number of points in the relative interior of $\cF$. (For a point both $\ell$ and $\ell^*$ are equal to 1.) Define functions $S_\cF$ and $T_\cF$ by
\begin{align*}
S_\cF(t)&=(1-t)^{d+1}\sum_{k=0}^\infty \ell(k\cF) t^k,\\
T_\cF(t)&=(1-t)^{d+1}\sum_{k=1}^\infty \ell^*(k\cF) t^k.
\end{align*}
We also set $S_\emptyset =1$.
\end{definition}

\begin{proposition} \label{prop:STproperties}
For $\cF\neq\emptyset$ the functions $S_\cF$ and $T_\cF$ defined above have the following properties:
\begin{enumerate}
\item $S_\cF(t)=t^{1+d} T_\cF(t^{-1})$.
\item $S_\cF(t)=1+[\ell(\cF)-d-1]t+\text{higher order terms}$.
\item $T_\cF(t)=\ell^*(\cF)t+[\ell^*(2\cdot\cF)-(d+1)\ell^*(\cF)]t^2+\text{higher order terms}$.
\item $S_\cF$ is a polynomial of degree at most $d$.
\item $T_\cF$ is a polynomial of degree exactly $d+1$.
\item $S_\cF$ has degree $d-r+1$ and $S_\cF(t)=t^{d-r+1} S_\cF(t^{-1})$ if  and only if $\cF$ is a Gorenstein polytope of index $r$.
\end{enumerate}
\end{proposition}

\begin{proof}
For 1 see~\cite{BatyrevBorisov1996a}, Proposition~3.6 and references there. The next two properties are immediate from the definition.  Then 4 and 5 follow from 1--3. For 6 see~\cite{BatyrevNill2007}, Remark~2.15 and references there.
\end{proof}

\begin{definition}
The generating function associated to the  dual Cayley cone $C$ of a nef partition is
\begin{align*}
E_C (u,v) = \frac{1}{(uv)^r}\sum_{\emptyset\leqslant x\leqslant y\leqslant P}
(-1)^{1+\dim x} u^{1+\dim y} S_x\left(\frac{v}{u}\right) S_{y^\vee}(uv) B_{[x,y]}\left(u^{-1},v\right),
\end{align*}
and its coefficients are the stringy Hodge numbers of the Calabi-Yau variety $X$ up to a sign:
\begin{align*}
E_C (u,v) = \sum_{p,q} (-1)^{p+q} h^{p,q}(X) u^p v^q\,.
\end{align*}
\end{definition}

\begin{remark}
The formula above is taken from~\cite{KRS2003}, the original one in~\cite{BatyrevBorisov1996a} is less convenient for actual computations since it includes infinite sums. A similar formula is also given in~\cite{BatyrevNill2007} (line~11 on page~57), but there is a typo~--- the posets of $B$-polynomials must be dualized.
\end{remark}

\begin{remark}
It is not obvious from the expression for $E_C$ that it is a polynomial of degree $2(n-r)$, although this is so for $C$ coming from a nef partition. On the other hand, the definition of $E_C$ makes sense for any $(n+r)$-dimensional Gorenstein cone of index $r$ and it is conjectured that it is always such a polynomial~\cite{BatyrevNill2007}.
\end{remark}

\section{The hypersurface case}

In this section we will derive a formula for $h^{1,1}$ of a hypersurface Calabi-Yau threefold in the four-dimensional toric variety ${\mathbb{P}}_\Delta$ using the generating function. While this formula can be obtained by other means, \cite{Batyrev1994}, it will serve as motivation and demonstration of techniques that will be used for nef partitions in the next section.

\begin{theorem} Let $\Delta$ be a four-dimensional reflexive polytope. Let $X\subset{\mathbb{P}}_\Delta$ be a generic anticanonical Calabi-Yau hypersurface. Then
\begin{align} \label{hyper:h11}
h^{1,1}(X)
&=\ell(\Delta)-5
-\sum_{\dim y=0}\ell^*(y^\vee)
+\sum_{\dim y=1} \ell^*(y)\cdot \ell^*(y^\vee),
\end{align}
where each sum runs over the faces of $\Delta^\circ$ of the indicated dimensions.
\end{theorem}

\begin{proof} 
A hypersurface can be treated in the above framework as a complete intersection with $r=1$, $P^*\simeq \Delta$, and $P\simeq\Delta^\circ$.

The generating function for $n=4$ and $r=1$ is given as
\begin{align*}
uv E_C (u,v) = \sum_{\emptyset\leqslant x\leqslant y\leqslant P}
(-1)^{1+\dim x} u^{1+\dim y} S_x\left(\frac{v}{u}\right) S_{y^\vee}(uv) B_{[x,y]}\left(u^{-1},v\right),
\end{align*}
so $h^{1,1}(X)=h^{2,2}(X)$ is equal to the coefficient of $u^2 v^2$ or $u^3 v^3$. Below, by extensively using Propositions~\ref{prop:Bproperties} and~\ref{prop:STproperties} without further mention, we will determine the coefficient of $u^3 v^3$ in the term corresponding to each pair $(x,y)$ on the right hand side. The reason for concentrating on a ``high $v$-degree'' term is that it allows us to deal only with simple $B$-polynomials corresponding to Eulerian posets of small rank, as we will see below. Note also that for the current case $\dim y^\vee=3-\dim y$ and $d[\emptyset, P]=5$.

First of all, observe that terms with $B$ depending on $v$ do not contribute to the coefficient of $u^3v^3$. Indeed, if the $v$-degree of $B$ is positive, then $d[x,y]\geqslant 3$ and we must have $\dim x\leqslant 1$, $\dim y\geqslant 2$, i.e. $\dim y^\vee\leqslant 1$, and at least one of these inequalities is strict. Then either both $S$-polynomials are equal to one or one is equal to one and the other is linear. On the other hand, $B_{[x,y]}\left(u^{-1},v\right)$ could only have $v$-degree $2$ or more if $d[x,y]\geqslant 5$, which is only possible for $[x,y]=[\emptyset, P]$, where both $S$-polynomials are equal to one. Therefore, the product of all these polynomials does not contain a $u^3v^3$ term.

Next we are going to consider cases with $d[x,y]\leqslant 2$ and either $x=\emptyset$ or $y=P$. The reason for separating these cases from the rest is that the degree of $S_\emptyset=S_{P^\vee}=1$ is not bounded by $\dim\emptyset=-1$.

\textbf{Suppose $x=\emptyset$ and $\dim y\leqslant 1$.} Then the corresponding term of the generating function is
\begin{align*}
u^{1+\dim y} S_{y^\vee}(uv) (1-u^{-1})^{1+\dim y}
=S_{y^\vee}(uv) (u-1)^{1+\dim y},
\end{align*}
where $S_{y^\vee}$ is a polynomial of degree at most $3-\dim y$. We see that in the only possible cases for $\dim y=-1,0$ the contribution to $u^3 v^3$-term is determined by the third degree term in $S_{y^\vee}$ and, by using the symmetry property of the $S$-polynomial of the reflexive polytope $P^*$, we obtain
\begin{align}
\label{hyper:empty_empty}
&\ell(P^*)-5,\\
\label{hyper:empty_0}
-\sum_{\dim y=0}&\ell^*(y^\vee).
\end{align}

\textbf{Suppose $\dim x\geqslant 2$ and $y = P$.} Then the corresponding term of the generating function is
\begin{align*}
(-1)^{1+\dim x} u^5 S_x\left(\frac{v}{u}\right) (1-u^{-1})^{4-\dim x},
\end{align*}
where $S_x$ is a polynomial of degree at most $\dim x$, which must be $3$ or $4$ in order to have any term with $v^3$. In these cases the contribution is determined by the third degree term in $S_x$, however, $\left(\dfrac{v}{u}\right)^3$ must be multiplied by $u^6$ in order to get $u^3 v^3$, which is not possible. We see that there are no contributions to $u^3 v^3$-term.

We consider remaining cases, $d[x,y]=0,1,2$, with $x\neq\emptyset$ and $y\neq P$.

\textbf{Suppose $x\neq\emptyset$, $y\neq P$, and $x=y$.} Then the corresponding term of the generating function is
\begin{align*}
(-u)^{1+\dim y} S_y\left(\frac{v}{u}\right) S_{y^\vee}(uv),
\end{align*}
where $S_y$ and $S_{y^\vee}$ are polynomials of degrees at most $\dim y$ and $3-\dim y$. In order to get $v^3$, we need to multiply the leading terms of these polynomials. The $u$-degree of the $v^3$ term in the total product will be $1+\dim y - \dim y + 3-\dim y = 4-\dim y$. Since we are interested in $u^3$ terms, we must have $\dim y=1$. The corresponding contribution is
\begin{align} \label{hyper:1_1}
\sum_{\dim y=1}& \ell^*(y)\cdot \ell^*(y^\vee).
\end{align}

\textbf{Suppose $x\neq\emptyset$, $y\neq P$, and $\dim y=1+\dim x$ or $\dim y=2+\dim x$.} Then we see that there are no contributions to $h^{1,1}(X)$, since the total degree of the product of the $S$-polynomials is at most 2.

Now combining all contributions we obtain \eqref{hyper:h11}, which completes the proof.
\end{proof}

\begin{remark}
The terms of~\eqref{hyper:h11} have the following meaning. Torus-invariant divisors of the ambient space, corresponding to lattice points of $\Delta$, except for the origin, have 4 linear relations between them. Divisors corresponding to the interior points of the facets do not intersect a generic Calabi-Yau hypersurface, while divisors corresponding to the interior points of faces of codimension two may become reducible when intersected with this hypersurface.
\end{remark}

\begin{corollary}
If $h^{1,1}=1$ for a Calabi-Yau hypersurface in the toric variety associated to a 4-dimensional reflexive polytope $\Delta$, then $\Delta$ is a simplex.
\end{corollary}

\begin{proof}
This easily follows from~\eqref{hyper:h11}, if we split $\ell(\Delta)$ into the sum of internal points of all of its faces:
\begin{align*}
h^{1,1}(X)
&=
\ell^*(\Delta)
+\sum_{\dim y=0,1,2,3}\ell^*(y^\vee)
-5
-\sum_{\dim y=0}\ell^*(y^\vee)
+\sum_{\dim y=1} \ell^*(y)\cdot \ell^*(y^\vee)\\
&=
\sum_{\dim y=1,2}\ell^*(y^\vee)
+\left[\sum_{\dim y=3}\ell^*(y^\vee)
-4\right]
+\sum_{\dim y=1} \ell^*(y)\cdot \ell^*(y^\vee).
\end{align*}
Since faces dual to faces of dimension 3 are vertices of $\Delta$, we see that the term in brackets is positive while all others are non-negative, and if $h^{1,1}(X)=1$, $\Delta$ must have exactly 5 vertices, i.e. be a simplex.
\end{proof}

\begin{remark}
While the number of reflexive polytopes of any fixed dimension is finite (up to $GL(\ZZ)$ action) and there is an algorithm allowing one to construct all of them (realized in PALP~\cite{PALP}), this number for dimension 5 and higher is so big, that it is practically impossible. However, results similar to the above corollary can allow for construction of all reflexive polytopes corresponding to Calabi-Yau varieties with small Hodge numbers.
\end{remark}

\section{The bipartite complete intersection case}

In this section we derive the closed form expression for $h^{1,1}(X)$ of a bipartite Calabi-Yau threefold complete intersection in the five-dimensional toric variety ${\mathbb{P}}_\Delta$.

\begin{theorem} \label{thm:h11cigeneric}
Let $\Delta$ be a five-dimensional reflexive polytope. Let $X\subset{\mathbb{P}}_\Delta$ be a complete intersection Calabi-Yau threefold corresponding to a fixed nef partition of $\Delta$ with associated dual Cayley polytope $P$. Then
\begin{alignat*}{3}
h^{1,1}(X)
&&=&
\ell(P^*)-7
&-&\!\!\sum_{\dim y=0}\left[ \ell^*(2\cdot y^\vee)-6\cdot\ell^*(y^\vee)\right]\\
&&+&\sum_{\dim y=1}\ell^*(y^\vee)
&+&\!\!\sum_{\dim y=1}\ell^*(y)\cdot \left[\ell^*(2\cdot y^\vee)-5\cdot\ell^*(y^\vee)\right]\\
&&-&\sum_{\dim y=2} [\ell(y) - \ell^*(y) - 3]\cdot \ell^*(y^\vee)
&-&\!\!\sum_{\sumfrac{\dim x=2}{\sumfrac{\dim y=3}{x<y}}} \ell^*(x)\cdot \ell^*(y^\vee)\\
&&+&\sum_{\dim y=3}\left[\ell^*(2\cdot y)-4\cdot\ell^*(y)\right]\cdot\ell^*(y^\vee),
\end{alignat*}
where sums run over faces of $P$ of indicated dimensions.
\end{theorem}

\begin{proof}
In this case we have the following relation for the generating function:
\begin{align*}
(uv)^2 E_C (u,v) = \sum_{\emptyset\leqslant x\leqslant y\leqslant P}
(-1)^{1+\dim x} u^{1+\dim y} S_x\left(\frac{v}{u}\right) S_{y^\vee}(uv) B_{[x,y]}\left(u^{-1},v\right),
\end{align*}
so $h^{1,1}(X)=h^{2,2}(X)$ is equal to the coefficient of $u^3 v^3$ or $u^4 v^4$. Below we will determine the coefficient of $u^4 v^4$ in the term corresponding to each pair $(x,y)$ on the right hand side. Note, that $\dim y^\vee=5-\dim y$.

First of all, let's consider all $(x,y)$-pairs with positive $v$-degree of $B_{[x,y]}$. Since $d[\emptyset,P]=7$, the highest possible $v$-degree of $B_{[x,y]}$ is 3. However, this is the only pair when degree 3 is a possibility and $S_\emptyset=S_{P^\vee}=1$, thus it does not give a contribution to $u^4v^4$. If the $v$-degree of $B_{[x,y]}$ is 2, then $d[x,y]\geqslant 5$ and either both $S$-polynomials are equal to 1, or one of them is 1 and the other is linear, so again such pairs yield no contribution to $u^4 v^4$. If the $v$-degree of $B$ is 1, then $d[x,y]\geqslant 3$ and it is easy to see that only for $x=\emptyset$ and $\dim y =2$ or $\dim x = 3$ and $y=P$ it is possible to have the total $v$-degree of $S_x$ and $S_{y^\vee}$ greater than 2. 

\textbf{Suppose $x=\emptyset$ and $\dim y=2$.} Let $k(y)$ be the number of vertices of $y$. Then the corresponding term of the generating function is
\begin{multline*}
u^3 S_{y^\vee}(uv) (1+[k(y)-(k(y)-3)v](u^{-2}-u^{-1})-u^{-3})\\
=S_{y^\vee}(uv) (u^3+[k(y)-(k(y)-3)v](u-u^2)-1),
\end{multline*}
where $S_{y^\vee}$ is a polynomial of degree at most 3. Its leading coefficient is $\ell^*(y^\vee)$, thus the contribution to $h^{1,1}$ is
\begin{align} \label{bipart:empty_2}
-\sum_{\dim y=2} (k(y)-3)\ell^*(y^\vee).
\end{align}

\textbf{Suppose $\dim x=3$ and $y=P$.}  Let $k(x)$ be the number of vertices of $x^\vee$. Then the corresponding term of the generating function is
\begin{multline*}
u^7 S_x\left(\frac{v}{u}\right) (1+[k(x)-(k(x)-3)v](u^{-2}-u^{-1})-u^{-3})\\
=S_x(uv) (u^7+[k(x)-(k(x)-3)v](u^5-u^6)-u^4),
\end{multline*}
where $S_x$ is a polynomial of degree at most 3. We see that terms with $v^4$ have $u$-degree 2 or 3, thus there is no contribution into $h^{1,1}$.

Next we are going to consider cases with $d[x,y]\leqslant 2$ and either $x=\emptyset$ or $y=P$. The reason for separating these cases from the rest is that the degree of $S_\emptyset=S_{P^\vee}=1$ is not bounded by $\dim\emptyset=-1$.

\textbf{Suppose $x=\emptyset$ and $\dim y\leqslant 1$.} Then the corresponding term of the generating function is
\begin{align*}
u^{1+\dim y} S_{y^\vee}(uv) (1-u^{-1})^{1+\dim y}
=S_{y^\vee}(uv) (u-1)^{1+\dim y},
\end{align*}
where $S_{y^\vee}$ is a polynomial of degree at most $5-\dim y$. We see that in all three cases for $\dim y=-1,0,1$ the contribution to $h^{1,1}$ is determined by the fourth degree term in $S_{y^\vee}$ and, using that $P^*$ is a six-dimensional Gorenstein polytope of index 2, we obtain
\begin{align}
\label{bipart:empty_empty}
&\ell(P^*)-7,\\
\label{bipart:empty_0}
-\sum_{\dim y=0}&\left[ \ell^*(2\cdot y^\vee)-6\cdot\ell^*(y^\vee)\right],\\
\label{bipart:empty_1}
\sum_{\dim y=1}&\ell^*(y^\vee).
\end{align}

\textbf{Suppose $\dim x\geqslant 4$ and $y = P$.} Then the corresponding term of the generating function is
\begin{align*}
(-1)^{1+\dim x} u^7 S_x\left(\frac{v}{u}\right) (1-u^{-1})^{6-\dim x},
\end{align*}
where $S_x$ is a polynomial of degree at most $\dim x$. We see that in all three cases for $\dim x=4,5,6$ the contribution to $h^{1,1}$ is determined by the fourth degree term in $S_x$, however $\left(\dfrac{v}{u}\right)^4$ must be multiplied by $u^8$ in order to get $u^4 v^4$. We see that this is not possible and there are no contributions to $h^{1,1}$.

We consider remaining cases, $d[x,y]=0,1,2$, with $x\neq\emptyset$ and $y\neq P$.

\textbf{Suppose $x\neq\emptyset$, $y\neq P$, and $x=y$.} Then the corresponding term of the generating function is
\begin{align*}
(-u)^{1+\dim y} S_y\left(\frac{v}{u}\right) S_{y^\vee}(uv),
\end{align*}
where $S_y$ and $S_{y^\vee}$ are polynomials of degrees at most $\dim y$ and $5-\dim y$. Consider degree $\alpha$ term in $S_y$, $0\leqslant\alpha\leqslant\dim y$. In order to get $v^4$ we need to multiply it by degree $4-\alpha$ term from $S_{y^\vee}$, $0\leqslant 4-\alpha\leqslant 5-\dim y$. Then the $u$-degree of this product with $u^{1+\dim y}$ will be $1+\dim y - \alpha + 4-\alpha = 5+\dim y-2\alpha$. Since we are interested in $u^4$ terms, possible values for $(\dim y,\alpha)$ satisfying all the restrictions are (1,1) and (3,2). The corresponding contributions are
\begin{align}
\label{bipart:1_1}
\sum_{\dim y=1}& \ell^*(y)\cdot \left[\ell^*(2\cdot y^\vee)-5\cdot\ell^*(y^\vee)\right],\\
\label{bipart:3_3}
\sum_{\dim y=3}& \left[\ell^*(2\cdot y)-4\cdot\ell^*(y)\right]\cdot\ell^*(y^\vee).
\end{align}

\textbf{Suppose $x\neq\emptyset$, $y\neq P$, and $\dim y=1+\dim x$.} Then the corresponding term of the generating function is
\begin{align*}
(-1)^{1+\dim x} u^{1+\dim y} S_x\left(\frac{v}{u}\right) S_{y^\vee}(uv) (1-u^{-1})
\!=\!
 (-u)^{1+\dim x} S_x\left(\frac{v}{u}\right) S_{y^\vee}(uv) (u-1),
\end{align*}
where $S_x$ and $S_{y^\vee}$ are polynomials of degrees at most $\dim x$ and $4-\dim x$, thus only the product of their leading terms yields terms with $v^4$. Taking into account remaining factors of the product, we see that possible $u$-degrees of terms with $v^4$ are $1+\dim x - \dim x + 4-\dim x = 5-\dim x$ and greater by one, $6-\dim x$. Therefore, we get $u^4 v^4$ terms and contributions to $h^{1,1}$ only for $\dim x=1$ or $\dim x=2$:
\begin{align}
\label{bipart:1_2}
-\sum_{\dim x=1, \dim y=2, x<y}& \ell^*(x)\cdot \ell^*(y^\vee),\\
\label{bipart:2_3}
-\sum_{\dim x=2, \dim y=3, x<y}& \ell^*(x)\cdot \ell^*(y^\vee).
\end{align}

\textbf{Suppose $x\neq\emptyset$, $y\neq P$, and $\dim y=2+\dim x$.} Then we see that there are no contributions to $h^{1,1}$, since the total degree of $S$-polynomials is at most 3.

Now let's combine all the contributions:
\begin{alignat*}{3}
h^{1,1}
&&=&
\ell(P^*)-7
&-&\sum_{\dim y=0}\left[ \ell^*(2\cdot y^\vee)-6\cdot\ell^*(y^\vee)\right]\\
&&+&\sum_{\dim y=1}\ell^*(y^\vee)
&-&\sum_{\dim y=2}(k(y)-3)\ell^*(y^\vee)\\
&&+&\sum_{\dim y=1}\ell^*(y)\cdot \left[\ell^*(2\cdot y^\vee)-5\cdot\ell^*(y^\vee)\right]
&-&\sum_{\sumfrac{\dim x=1}{\sumfrac{\dim y=2}{x<y}}} \ell^*(x)\cdot \ell^*(y^\vee)\\
&&-&\sum_{\sumfrac{\dim x=2}{\sumfrac{\dim y=3}{x<y}}} \ell^*(x)\cdot \ell^*(y^\vee)
&+&\sum_{\dim y=3}\left[\ell^*(2\cdot y)-4\cdot\ell^*(y)\right]\cdot\ell^*(y^\vee).
\end{alignat*}
Observe, that terms \eqref{bipart:empty_2} and \eqref{bipart:1_2} can be naturally combined, since the first one contains the number of vertices of a 2-face $y$, while the second one sums over internal points of all edges of each 2-face $y$. Therefore, the total contribution of these two terms is
\begin{align*}
-\sum_{\dim y=2} [\ell(y) - \ell^*(y) - 3]\cdot \ell^*(y^\vee),
\end{align*}
where $\ell(y)-\ell^*(y)$ is the number of boundary points of $y$. This leads us to the stated formula for $h^{1,1}(X)$ and completes the proof.
\end{proof}

\begin{lemma} \label{lem:otherhpq}
In the notation of Theorem~\ref{thm:h11cigeneric}, we have 
\begin{align*}
h^{3,3}(X)
&=
1
+\sum_{\dim y=1}\ell^*(y)\cdot\ell^*(y^\vee)
-\sum_{\dim y=0}\ell^*(y^\vee),\\
h^{2,3}(X)
&=
\sum_{\dim y=2}\ell^*(y)\cdot\ell^*(y^\vee),\\
h^{3,2}(X)
&=
-\sum_{\dim y=2}
\left[\ell(y)+3\cdot\ell^*(y)-3-\ell^*(2\cdot y)\right] \ell^*(y^\vee).
\end{align*}
\end{lemma}

\begin{proof}
The same type of argument as for the proof of the theorem (but shorter).
\end{proof}

\begin{corollary}
In the notation of Theorem~\ref{thm:h11cigeneric}, if the nef partition is indecomposable, the following relations hold:
\begin{align}
\label{h33is1}
\sum_{\dim y=1}\ell^*(y)\cdot\ell^*(y^\vee)
&=\sum_{\dim y=0}\ell^*(y^\vee),\\
\label{h23is0}
\sum_{\dim y=2}\ell^*(y)\cdot\ell^*(y^\vee)
&=0,\\
\label{h32is0}
\sum_{\dim y=2} \ell^*(2\cdot y)\cdot \ell^*(y^\vee)
&=\sum_{\dim y=2} \left[\ell(y)-3\right] \ell^*(y^\vee).
\end{align}
\end{corollary}

\begin{proof}
Follows immediately from Lemma~\ref{lem:otherhpq}, since we know that $h^{3,3}(X)=1$ and $h^{2,3}(X)=h^{3,2}(X)=0$.
\end{proof}

Now we use this corollary to prove the following result.

\begin{lemma} \label{lsylsyvis0}
Let $\Delta$ be a five-dimensional reflexive polytope. Let $P$ be the dual Cayley polytope of an indecomposable two part nef partition of $\Delta$. If $y$ is a face of $P$, then $\ell^*(y)\cdot\ell^*(y^\vee)=0$.
\end{lemma}
\begin{proof}
First, let $y$ be a vertex. Then $\ell^*(y)=1$ and we need to show that $\ell^*(y^\vee)=0$. Note that $y^\vee$ is a $5$-dimensional facet of $P^*$. Then either $y^\vee$ is one of the polytopes $\nabla_1$ or $\nabla_2$ of the dual nef partition and it does not have an interior point, since the nef partition is indecomposable (\cite{BatyrevNill2007}, Corollary~6.12), or $y^\vee$ has non-empty intersection with both $\nabla_1$ and $\nabla_2$. In the latter case consider the projection of $N_\RR\times\RR^2\supset P^*$ onto the second factor. Then the image of $y^\vee$ is the line segment from $(1,0)$ to $(0,1)$, which does not have interior points. Therefore, in any case $\ell^*(y^\vee)=0$ as desired.

Now from relations \eqref{h33is1} and \eqref{h23is0} we conclude the result for $\dim y\leqslant 2$, but then for $\dim y\geqslant 3$ it follows by symmetry.
\end{proof}

\begin{remark}
Relation~\eqref{h32is0} follows from Lemma~\ref{lsylsyvis0} and Pick's formula. Indeed, let $y$ be a face of $P$ of dimension $2$, such that $\ell^*(y^\vee)\neq 0$, then we know that $\ell^*(y)=0$. Then the area of $y$ is $A(y)=\ell(y)/2-1$ and
\begin{align*}
 A(2\cdot y)
&=\ell^*(2\cdot y) + \dfrac{\ell(2\cdot y)-\ell^*(2\cdot y)}{2} - 1
=4A(y)=2\ell(y)-4,
\end{align*}
but the number of boundary points of $2\cdot y$ is $2\ell(y)$, thus  $\ell^*(2\cdot y)=\ell(y)-3$.
\end{remark}

\begin{theorem} \label{thm:h11cispecial}
Let $\Delta$ be a five-dimensional reflexive polytope. Let $X\subset{\mathbb{P}}_\Delta$ be a complete intersection Calabi-Yau threefold corresponding to a fixed \emph{indecomposable} nef partition of $\Delta$ with associated dual Cayley polytope $P$. Then
\begin{alignat*}{3}
h^{1,1}(X)\ 
&&=&\ 
\ell(P^*)-7
-\sum_{\dim y=0} \ell^*(2\cdot y^\vee)\ 
&+&\sum_{\dim y=1}\ell^*(y^\vee)\\
&&+&\sum_{\dim y=1}\ell^*(y)\cdot \ell^*(2\cdot y^\vee)
&-&\sum_{\sumfrac{\dim x=2}{\sumfrac{\dim y=3}{x<y}}} \ell^*(x)\cdot \ell^*(y^\vee)\\
&&-&\sum_{\dim y=2}\ell^*(2\cdot y)\cdot \ell^*(y^\vee)
&+&\sum_{\dim y=3}\ell^*(2\cdot y)\cdot \ell^*(y^\vee).
\end{alignat*}
where sums run over faces of $P$ of indicated dimensions.
\end{theorem}

\begin{proof}
Follows from Theorem~\ref{thm:h11cigeneric} and Lemma~\ref{lsylsyvis0}.
\end{proof}

\section{Relations with other results}

It would be desirable to have a geometric interpretation for \emph{each term} of the obtained expressions for Hodge numbers and, in particular, to be able to identify the toric component of $h^{1,1}(X)$, given by images of the toric-invariant divisors of the ambient space or, equivalently, the polynomial part of $h^{2,1}(X)$, corresponding to polynomial deformations of the complete intersection in the ambient space. (In the hypersurface case this extra information follows ``for free'' from the proof of Batyrev's formulas for the Hodge numbers.) While there is an algorithm for computing the toric part of the cohomology ring (see~\cite{BKOS2007}, for example), it does not give directly a ``closed form'' expression for its dimension. Also Borisov and Mavlyutov have constructed complete stringy cohomology spaces in~\cite{BorisovMavlyutov2003} for semiample hypersurfaces in toric varieties and perhaps their techniques may be used in complete intersection case as well.

It would also be interesting to compare the result of Theorem~\ref{thm:h11cispecial} with the previously known formulas for Hodge numbers of complete intersections obtained by Batyrev and Borisov in~\cite{BatyrevBorisov1996b}. They have considered a special case when all divisors corresponding to the nef partition in the non-resolved variety (i.e. the variety corresponding to $\Delta$ without triangulation of the boundary) are ample.\footnote{In~\cite{DoranMorgan2007} the first author and John Morgan relate the closed form expressions of~\cite{BatyrevBorisov1996b} directly to the (mixed) Hodge structure on middle-dimensional cohomology for complete intersection Calabi-Yau threefolds in toric varieties.  A generalization that allows for geometric interpretations for each term and identification of the toric components of $h^{1,1}(X)$ would be most useful for such applications.} Their formulas restricted to our case are given below, although we were not yet able to match all terms with ours.

\begin{definition}
A lattice polytope $\Delta'$ is a Minkowski summand of another lattice polytope $\Delta$ if there exist $\mu\in\ZZ_{>0}$ and a lattice polytope $\Delta''$ such that $\mu\Delta=\Delta'+\Delta''$.
\end{definition}

If the divisors given by polytopes $\nabla_i$ are ample in $\PP_\Delta$ (before partial resolution corresponding to a triangulation of $\partial\Delta$), then $\Delta^\circ$ is a Minkowski summand of $\nabla_i$ for all $i$, all these polytopes are combinatorially equivalent, and each face $\theta$ of $\Delta^\circ$ decomposes into Minkowski sum $\theta=\sum_i \theta_i$, where $\theta_i$ is a face of $\nabla_i$ of the same dimension as $\theta$. In this case, the nef-partition is necessarily irreducible and Theorem~\ref{thm:h11cispecial} is applicable. Another way to compute Hodge numbers in this case is the following result.

\begin{corollary}[from Corollary~8.4\cite{BatyrevBorisov1996b}]
Let $\Delta$ be a five-dimensional reflexive polytope. Let $X\subset{\mathbb{P}}_\Delta$ be a complete intersection Calabi-Yau threefold corresponding to a fixed nef partition of $\Delta$ with ample divisors corresponding to polytopes $\nabla_i$ of the dual nef partition. Then
\begin{align*}
\allowdisplaybreaks
h^{1,1}(X)
&=\ell(\Delta)-6
-\sum_{\dim \theta=4} \ell^*(\theta)
-\sum_{\dim \theta=3} \ell^*(\theta)\\
&+\sum_{\dim \theta=2} \ell^*(\theta) \cdot
\big[\ell^*(\theta^*) - \ell^*(\theta_1^*) - \ell^*(\theta_2^*)\big],\displaybreak[0]\\ 
h^{2,1}(X)
&=\big[\ell^*(2\nabla_1+\nabla_2)-\ell^*(2\nabla_1)
+\ell^*(\nabla_1+2\nabla_2)-\ell^*(2\nabla_2) \big] - 7\\
&-\sum_{\dim \theta=0} \big[\ell^*(\theta^*) - \ell^*(\theta_1^*) - \ell^*(\theta_2^*)\big]\\
&+\sum_{\dim \theta=1} \ell^*(\theta) \cdot
\big[\ell^*(\theta^*) - \ell^*(\theta_1^*) - \ell^*(\theta_2^*)\big],
\end{align*}
where the sums are over faces of $\Delta$ of indicated dimensions, $\theta^*$ is the face of $\Delta^\circ$ dual to $\theta$, and $\theta^*=\theta_1^*+\theta_2^*$ is the decomposition into Minkowski sum with $\theta_i^*$ being a face of $\nabla_i$.
\end{corollary}
\begin{proof}
Follows from Corollary~8.4 in~\cite{BatyrevBorisov1996b}, after restricting to $\dim\Delta=5$ and ${\dim X=3}$.
\end{proof}

\section{Examples}

In this section we apply our formula to several simplices and non-simplices $\Delta$ with the corresponding $h^{1,1}(X)$ equal and not equal to 1. The reflexive polytopes considered here were taken from the data supplements to~\cite{KKRS2005}.\footnote{Available at {\tt http://hep.itp.tuwien.ac.at/$\sim$kreuzer/CY/hep-th/0410018.html}}

\begin{example}[Simplex, $h^{1,1}(X)=1$]
Let vertices of $\Delta$ be given by columns of the matrix
\begin{align*}
\begin{pmatrix}
0 & -3 & 0 & 1 & 0 & 0 \\
0 & -3 & 1 & 0 & 0 & 0 \\
0 & -2 & 0 & 0 & 0 & 1 \\
0 & -2 & 0 & 0 & 1 & 0 \\
1 & -1 & 0 & 0 & 0 & 0
\end{pmatrix}
\end{align*}
and consider the nef partition corresponding to $V_1=\set{1,3,5}$ and $V_2=\set{2,4,6}$. Applying Theorem~\ref{thm:h11cispecial}, we get (each sum is written as a separate term)
\begin{align*}
h^{1,1}(X) &= 8 - 7 - 0 + 0 + 0 - 0 - 0 + 0 = 1,\\
h^{2,1}(X) &= 98 - 7 - 30 + 0 + 0 - 0 - 0 + 0 = 61.
\end{align*}
\end{example}

\begin{example}[Non-simplex, $h^{1,1}(X)=1$]
Let vertices of $\Delta$ be given by columns of the matrix
\begin{align*}
\begin{pmatrix}
1 & 1 & -1 & 0 & 0 & 0 & 1 \\
1 & -3 & 0 & 1 & 0 & 0 & 0 \\
0 & -1 & 0 & 0 & 1 & 0 & 0 \\
1 & -2 & 0 & 0 & 0 & 1 & 0 \\
2 & -2 & 0 & 0 & 0 & 0 & 0
\end{pmatrix}
\end{align*}
and consider the nef partition corresponding to $V_1=\set{1,2,5,6}$ and $V_2=\set{3,4,7}$. Applying Theorem~\ref{thm:h11cispecial}, we get (each sum is written as a separate term)
\begin{align*}
h^{1,1}(X) &= 9 - 7 - 1 + 0 + 0 - 0 - 0 + 0 = 1,\\
h^{2,1}(X) &= 54 - 7 - 14 + 0 + 5 - 0 - 1 + 0 = 37.
\end{align*}
\end{example}

\begin{example}[Non-simplex, $h^{1,1}(X)\neq 1$]
Let vertices of $\Delta$ be given by columns of the matrix
\begin{align*}
\begin{pmatrix}
-1 & 1 & 0 & 0 & 0 & 1 & -2 & -1 \\
-1 & 0 & 1 & 0 & 0 & 0 & -1 & 0 \\
0 & 1 & 0 & 1 & 0 & 0 & -1 & 0 \\
-1 & 1 & 0 & 0 & 1 & 0 & -1 & 0 \\
-2 & 2 & 0 & 0 & 0 & 0 & 0 & 0
\end{pmatrix}
\end{align*}
and consider the nef partition corresponding to $V_1=\set{1,3,7}$ and $V_2=\set{2,4,5,6,8}$. Applying Theorem~\ref{thm:h11cispecial}, we get (each sum is written as a separate term)
\begin{align*}
h^{1,1}(X) &= 10 - 7 - 1 + 0 + 0 - 0 - 0 + 0 = 2,\\
h^{2,1}(X) &= 46 - 7 - 11 + 0 + 1 - 0 - 0 + 1 = 30.
\end{align*}
\end{example}

\begin{example}[Simplex, $h^{1,1}(X) \neq 1$]
Let vertices of $\Delta$ be given by columns of the matrix
\begin{align*}
\begin{pmatrix}
-1 & 0 & 0 & 0 & 1 & 0 \\
-1 & 0 & 1 & 0 & 0 & 0 \\
-1 & 0 & 0 & 1 & 0 & 0 \\
-2 & 1 & 0 & 0 & 0 & 1 \\
-2 & 2 & 0 & 0 & 0 & 0
\end{pmatrix}
\end{align*}
and consider the nef partition corresponding to $V_1=\set{2,6}$ and $V_2=\set{1,3,4,5}$. Applying Theorem~\ref{thm:h11cispecial}, we get (each sum is written as a separate term)
\begin{align*}
h^{1,1}(X) &= 9 - 7 - 0 + 0 + 0 - 0 - 0 + 0 = 2,\\
h^{2,1}(X) &= 83 - 7 - 27 + 0 + 10 - 0 - 1 + 0 = 58.
\end{align*}
\end{example}

\begin{remark}
As can be seen from these and a few other examples, none of the terms in our formula vanishes identically.
\end{remark}

\providecommand{\bysame}{\leavevmode\hbox to3em{\hrulefill}\thinspace}
\providecommand{\MR}{\relax\ifhmode\unskip\space\fi MR }
\providecommand{\MRhref}[2]{%
  \href{http://www.ams.org/mathscinet-getitem?mr=#1}{#2}
}
\providecommand{\href}[2]{#2}

\end{document}